\documentclass[a4paper,12pt]{article}
\makeatletter
\newcommand*{\rom}[1]{\expandafter\@slowromancap\romannumeral #1@}
\makeatother

\usepackage{lipsum,relsize,graphicx,amsmath,amsthm,amssymb,mathtools}
\usepackage[margin=1in]{geometry}
\usepackage{hyperref}

\newtheorem{theorem}{Theorem}[section]
\newtheorem{lemma}[theorem]{Lemma}
\newtheorem{corollary}[theorem]{Corollary}
\newtheorem{definition}{Definition}[section]
\newtheorem{proposition}[theorem]{Proposition}

\newcounter{tmp}

\begin{document}

\title{Flat dimension for power series over valuation rings}

\author{Adam Jones}

\date{\today}

\maketitle

\begin{abstract}
\noindent We examine the power series ring $R[[X]]$ over a valuation ring $R$ of rank 1, with proper, dense value group. We give a counterexample to Hilbert's syzygy theorem for $R[[X]]$, i.e. an $R[[X]]$-module $C$ that is flat over $R$ and has flat dimension at least 2 over $R[[X]]$, contradicting a previously published result. The key ingredient in our construction is an exploration of the valuation theory of $R[[X]]$. We also use this theory to give a new proof that $R[[X]]$ is not a coherent ring, a fact which is essential in our construction of the module $C$.
\end{abstract}

\tableofcontents

\section{Introduction}

Understanding the algebraic and homological properties of the power series ring $R[[X]]$ in a single variable over a ring $R$ is a difficult and ongoing problem, even when $R$ is assumed to be commutative. When $R$ is Noetherian, $R[[X]]$ is a Noetherian ring with associated graded ring $R[Y]$, so its properties are tractable, but in general very little can be concretely said. Even in the nicest non-Noetherian case when $R$ is a valuation ring, it is known that $R[[X]]$ rarely even satisfies the modest property of coherence.\\

\noindent With regard to homological properties, we would like to generalise some of the basic results on the polynomial ring $R[X]$. One particular desirable property, explored in \cite{flaw}, \cite{high-rank}, \cite{noetherian},  and elsewhere is Hilbert's syzygy theorem. To give some context, consider the following well known definition \cite[Definition 4.1.1, Lemma 4.1.6, Lemma 4.1.8]{Weibel}.

\begin{definition}

Let $R$ be a ring, we say that a left $R$-module $M$ has \emph{projective dimension $n$} (proj.dim$_R(M)=n$) if it satisfies either of the folloing equivalent definitions:

\begin{itemize}

\item There exists a projective resolution $0\to P_n\to\dots\to P_1\to P_0\to M\to 0$, and any projective resolution for $M$ has length at least $n$.

\item $Ext_R^{n+1}(M,N)=0$ for all $R$-modules $N$, and $Ext_R^{n}(M,N)\neq 0$ for some $R$-module $N$.

\end{itemize}

\noindent Similarly, we say that $M$ has \emph{flat dimension $n$} (f.dim$_R(M)=n$) if it satisfies either of the following equivalent conditions:

\begin{itemize}

\item There exists a flat resolution $0\to F_n\to\dots\to F_1\to F_0\to M\to 0$, and any flat resolution for $M$ has length at least $n$.

\item $Tor_{n+1}^{R}(M,N)=0$ for all $R$-modules $N$, and $Tor_n^{R}(M,N)\neq 0$ for some $R$-module $N$.

\end{itemize}

\noindent If no such $n$ exists then proj.dim$_R(M)=\infty$ and f.dim$_R(M)=\infty$. The \emph{(left) global dimension} of $R$, gl.dim$(R):=\sup\{$proj.dim$_R(M):M$ a left $R$-module$\}$, and the \emph{weak global dimension} of $R$, w.dim$(R):=\sup\{$f.dim$_R(M):M$ a left $R$-module$\}$.

\end{definition}

\noindent If a ring $R$ is Noetherian, then gl.dim$(R)=$ w.dim$(R)$ \cite[Proposition 4.1.5]{Weibel}, so these concepts only need to be explored separately for non-Noetherian rings. Hilbert's syzygy theorem states that for \emph{any} ring $R$, gl.dim$(R[X])=$ gl.dim$(R)+1$ and w.dim$(R[X])=$ w.dim$(R)+1$, and we would like this result to also hold over the power series ring $R[[X]]$. In the case when $R$ is Noetherian, $R[[X]]$ is also Noetherian, and do indeed have that gl.dim$(R[[X]])=$ gl.dim$(R)+1$ \cite[Theorem 2]{noetherian}. We would like to prove a similar result for the weak dimension of $R[[X]]$ when $R$ is a non-Noetherian ring.\\

\noindent It was proved in \cite[Lemma 1]{high-rank}, that in the case where $R[[X]]$ is a coherent ring (see Definition \ref{coh} below), then w.dim$(R[[X]])$ = w.dim$(R)+1$, but sadly this case is very rare. However, in \cite[Corollary 4.4]{flaw}, it was claimed by Bouchiba that the same result does indeed hold whenever the coefficient ring $R$ is coherent, in fact a stronger version of the syzygy was established; that if $M$ is an $R[[X]]$-module, then f.dim$_{R[[X]]}(M)\leq 1+$ f.dim$_R(M)$ \cite[Corollary 3.2]{flaw}. 

Unfortunately, however, there is a small error in this argument, specifically in the proof of \cite[Theorem 3.1]{flaw}, which claims the existence of an exact sequence of $R[[X]]$-modules $0\to R[[X]]\otimes_R M\to R[[X]]\otimes_R M\to M\to 0$ for any $R[[X]]$-module $M$, which is established by tensoring an exact sequence $0\to R[[X]]\otimes_R R[[X]]\to R[[X]]\otimes_R R[[X]]\to R[[X]]\to 0$ with $M$ over $R[[X]]$. The issue is that the latter exact sequence is only an exact sequence of right $R[[X]]$-modules, so tensoring on the the right by $M$ destroys the $R[[X]]$-module structure, and thus the former sequence is only an exact sequence of $R$-modules, not of $R[[X]]$-modules.\\

\noindent Furthermore, in this paper we can now confirm that \cite[Corollary 3.2]{flaw} is in fact false; in certain cases where $R$ is a coherent domain we can find examples of $R[[X]]$-modules $M$ such that f.dim$_{R[[X]]}(M)>$ f.dim$_R(M)+1$, which is summarised by our main result:

\begingroup
\setcounter{tmp}{\value{theorem}}
\setcounter{theorem}{0} 
\renewcommand\thetheorem{\Alph{theorem}}

\begin{theorem}\label{A}
Let $R$ be a valuation ring of rank 1, whose value group is a proper, dense subgroup of $\mathbb{R}$. Then there exists an $R[[X]]$-module $C$ such that f.dim$_R(C)=0$ and f.dim$_{R[[X]]}(C)\geq 2$.
\end{theorem}

\endgroup

\noindent Before we begin outlining the construction of $C$, let us first recall some important definitions and results.

\begin{definition}

Let $R$ be any ring, $\Gamma$ a totally ordered abelian group. A \emph{filtration} on $R$ is a map $w:R\to\Gamma\cup\{\infty\}$ satisfying for all $r,s\in R$:

\begin{itemize}

\item $w(r+s)\geq\min\{w(r),w(s)\}$.

\item $w(rs)\geq w(r)+w(s)$.

\item $w(1)=0$ and $w(0)=\infty$

\end{itemize}

\noindent Moreover, we say that $w$ is \emph{separated} if $w(r)=\infty$ implies that $r=0$, and we say that $w$ is a \emph{valuation} if $w$ is separated and $w(rs)=w(r)+w(s)$ for all $r,s\in R$.

\end{definition}

\begin{lemma}\label{localisation}
If $R$ is a commutative ring, $U$ is a multiplicatively closed subset of $R$, then the localisation $R_U$ is a flat $R$-module. Also, if $R$ carries a valuation $w$ then $R$ is a domain and $w$ extends uniquely to a valuation of $R_U$.
\end{lemma}

\begin{proof}

It is well known that $R_U$ is flat over $R$ (see e.g. \cite[Theorem 3.2.2]{Weibel}). If $w:R\to\Gamma\cup\{\infty\}$ is a valuation on $R$ and $rs=0$ for some $r,s\in R$, then $\infty=w(rs)=w(r)+w(s)$, so $w(r)=\infty$ or $w(s)=\infty$, i.e. $r=0$ or $s=0$ and $R$ is a domain. 

Moreover, define $w':R_U\to\Gamma\cup\{\infty\},ru^{-1}\mapsto w(r)-w(u)$, and this is well-defined since if $ru^{-1}=sv^{-1}$ then $rv=su$ so $w(r)-w(u)=w(s)-w(v)$, and it is straightforward to show that it is a valuation. This extension is unique, because if $v$ is any valuation on $R_U$ such that $v(r)=w(r)$ for all $r\in R$ then $v(ru^{-1})=v(r)-v(u)=w(r)-w(u)=w'(ru^{-1})$ for all $r\in R$, $u\in U$.\end{proof}

\noindent The following result is well known.

\begin{lemma}\label{val-ring}

For any commutative domain $R$, the following are equivalent.

\begin{itemize}
\item There exists a field $F$ with $R\subseteq F$ and a valuation $v:F\to\Gamma\cup\{\infty\}$ such that $R=\{x\in F:v(r)\geq 0\}$.

\item For all ideals $I,J$ of $R$, $I\subseteq J$ or $J\subseteq I$. In particular, all finitely generated ideals are principal.

\item For all $x\in Q(R)$, $x\in R$ or $x^{-1}\in R$.
\end{itemize}

\noindent If $R$ satisfies any of these conditions, we say that $R$ is a \emph{valuation ring}, and we call the group $v(F\backslash\{0\})$ the \emph{value group} of $R$.

\end{lemma}

\noindent It is also well known that a valuation ring $R$ has dimension 1 if and only if its value group is a subset of $\mathbb{R}$ with its usual ordering, in which case we say that $R$ has \emph{rank 1}. In this case, the value group is either discrete or else it is dense in $\mathbb{R}$. It is straightforward to see that a valuation ring $R$ is Noetherian if and only if it has rank 1 and the value group is discrete, but there is a related property that is satisfied by general valuation rings:

\begin{definition}\label{coh}

A ring $R$ is \emph{coherent} if for any finitely generated one-sided ideal $I$ of $R$, $I$ is finitely presented, i.e. there exists an exact sequence of $R$-modules $R^m\to R^n\to I\to 0$ for some $m,n\in\mathbb{N}$.

\end{definition}

\noindent Clearly any Noetherian ring is coherent, and since all finitely generated ideals in a valuation ring are principal, it follows that any valuation ring is coherent. It is also known that if $R$ is a valuation ring, then the polynomial ring $R[X]$ is coherent \cite[Theorem 7.3.3]{Glaz}, but sadly there is no similar result for the power series ring $R[[X]]$.

It was proved in \cite[Theorem 1]{high-rank} that if $R$ has rank greater than 1, then $R[[X]]$ is not coherent, and it was shown in \cite[Corollary Section 3]{rank-1} that if $R$ has rank 1 and the value group is dense in $\mathbb{R}$, but not equal to $\mathbb{R}$, then $R[[X]]$ is also not coherent. It is not currently known whether $R[[X]]$ can be coherent when $R$ has value group $\mathbb{R}$.\\

\noindent From now on, we will assume that all filtrations/valuations take values in $\mathbb{R}\cup\{\infty\}$. In particular, if $R$ is a valuation ring we will assume it has rank 1, and we will usually assume that its value group is a proper, dense subgroup of $\mathbb{R}$. In section 2, we will explore some valuation theory, and prove that there is a canonical extension of the valuation on $R$ to the power series ring $R[[X]]$. In section 3, we will use this valuation to give an alternative proof that $R[[X]]$ is not a coherent ring.

Following this, in section 4, we will consider some well-behaved localisations of $R[[X]]$, before using them and some homological algebra in section 5 to construct the $R[[X]]$-module $C$ we need in our main theorem, and the construction of this module will depend strongly on the incoherence of $R[[X]]$; indeed if such a module $C$ did not exist it would follow that $R[[X]]$ was coherent.\\

\noindent\textbf{Acknowledgments:} I am extremely grateful to Samir Bouchiba for several very helpful and fruitful exchanges regarding the arguments in his paper \cite{flaw}. I would also like to thank the Heilbronn Institute for Mathematical Research for funding and supporting this research.

\section{Valuation theory for power series rings}

Throughout this section, we will let $R$ be any commutative ring, and let $v:R\to\mathbb{R}\cup\{\infty\}$ be a valuation such that $v(r)\geq 0$ for all $r\in R$. We want to explore different ways of extending $v$ to the power series ring $R[[X]]$.\\

\noindent For each $\lambda\geq 0$, define a map:

\begin{equation}
v_{\lambda}:R[[X]]\to\mathbb{R}\cup\{\infty\}, \underset{n\in\mathbb{N}}{\sum}{r_nX^n}\mapsto\inf\{v(r_n)+\lambda n:n\in\mathbb{N}\}
\end{equation}

\begin{lemma}\label{attained}
If $\lambda>0$ then $v_{\lambda}\left(\underset{n\in\mathbb{N}}{\sum}{r_nX^n}\right)=v(r_n)+\lambda n$ for some $n$, and there are only finitely many such $n$.
\end{lemma}

\begin{proof}

Let $m$ be minimal such that $r_m\neq 0$, and choose $k> m$ minimal such that $v(r_m)<(k-m)\lambda$, and hence $v(r_m)<(k'-m)\lambda$ for all $k'\geq k$.\\

\noindent Thus $v(r_{k'})+\lambda k'=v(r_{k'})+\lambda(k'-m)+\lambda m> v(r_m)+\lambda m$ for all $k'\geq k$, and it follows that $\inf\{v(r_n)+\lambda n:n\in\mathbb{N}\}=\min\{v(r_n)+\lambda n:m\leq n <k\}$.

Moreover, since $v(r_k')+\lambda k'>v(r_m)+\lambda m\geq\min\{v(r_n)+\lambda n:m\leq n <k\}$ for all $k'\geq k$, it follows that the infimum is attained only in the finite set $\{n\in\mathbb{N}:m\leq n <k\}$.\end{proof}

\noindent Note that this lemma is false if $\lambda=0$ and $v$ is not discrete.

\begin{lemma}\label{filtration}
For each $\lambda\geq 0$, $v_{\lambda}$ is a separated ring filtration.
\end{lemma}

\begin{proof}

We need to show that for all $f,g\in R[[X]]$, $v_{\lambda}(f+g)\geq\min\{v_{\lambda}(f),v_{\lambda}(g)\}$ and $v_{\lambda}(fg)\geq v_{\lambda}(f)+v_{\lambda}(g)$. Suppose that $f=\underset{n\in\mathbb{N}}{\sum}{r_nX^n}$ and $g=\underset{n\in\mathbb{N}}{\sum}{s_nX^n}$.\\

\noindent Then $v_{\lambda}(f+g)=v_{\lambda}\left(\underset{n\in\mathbb{N}}{\sum}{(r_n+s_n)X^n}\right)=\inf\{v(r_n+s_n)+\lambda n:n\in\mathbb{N}\}$

$\geq\inf\{\min\{v(r_n), v(s_n)\}+\lambda n:n\in\mathbb{N}\}=\inf\{\min\{v(r_n)+\lambda n, v(s_n)+\lambda n\}:n\in\mathbb{N}\}$\\

$\geq\min\{\inf\{v(s_n)+\lambda n:n\in\mathbb{N}\},\inf\{v(r_n)+\lambda n:n\in\mathbb{N}\}\}=\min\{v_{\lambda}(f),v_{\lambda}(g)\}$.\\

\noindent Also, $v_{\lambda}(fg)=v_{\lambda}\left(\underset{n\in\mathbb{N}}{\sum}{\left(\underset{i+j=n}{\sum}{r_is_j}\right)X^n}\right)=\inf\left\{v\left(\underset{i+j=n}{\sum}{r_is_j}\right)+\lambda n:n\in\mathbb{N}\right\}$

$\geq\inf\{\min\{v(r_i)+v(s_j):i+j=n\}+\lambda n:n\in\mathbb{N}\}$\\

$=\inf\{\min\{v(r_i)+v(s_j)+\lambda n:i+j=n\}:n\in\mathbb{N}\}$\\

$\geq \inf\{\min\{v(r_i)+\lambda i+v(s_j)+\lambda j: i,j\in\mathbb{N}\}$\\

$=\inf\{\min\{v(r_i)+\lambda i: i\in\mathbb{N}\}+\inf\{\min\{v(s_j)+\lambda j: j\in\mathbb{N}\}=v_{\lambda}(f)+v_{\lambda}(g)$.\\

\noindent Also, $v_{\lambda}(f)=\infty$ if and only if $v(r_n)+\lambda n=\infty$ for all $n$, which is true if and only if $r_n=0$ for all $n$ and $f=0$, while $v_{\lambda}(1)=v(1)=0$, so $v_{\lambda}$ is a separated filtration.\end{proof}

\noindent So, for each $\lambda,\alpha\geq 0$, let $F_{\alpha}^{\lambda}R[[X]]:=\{f\in R[[X]]:v_{\lambda}(f)\geq\alpha\}$ and $F_{\alpha^+}^{\lambda}R[[X]]:=\{f\in R[[X]]:v_{\lambda}(f)>\alpha\}$. Then $F_{\alpha^+}^{\lambda}R[[X]]$ is an additive subgroup of $F_{\alpha}^{\lambda}R[[X]]$, so we can define the \emph{associated graded ring}:

\begin{equation}
\text{gr}_{\lambda}\text{ }R[[X]]:=\underset{\alpha\geq 0}{\bigoplus}{\frac{F_{\alpha}^{\lambda}R[[X]]}{F_{\alpha^+}^{\lambda}R[[X]]}}
\end{equation}

\noindent This is clearly an $\mathbb{R}_{\geq 0}$-graded abelian group, and it carries a ring structure defined by $(f+F_{\alpha^+}^{\lambda}R[[X]])\cdot (g+F_{\beta^+}^{\lambda}R[[X]])=fg+F_{(\alpha+\beta)^+}^{\lambda}R[[X]]$.

\begin{proposition}\label{valuation1}
For each $\lambda>0$, there exists an isomorphism $\Theta_{\lambda}:\text{gr}_{\lambda}\text{ }R[[X]]\to (\text{gr }R)[Y]$, and it follows that $v_{\lambda}$ is a valuation.
\end{proposition}

\begin{proof}

Define $\Theta_{\lambda}:\text{gr}_{\lambda}\text{ }R[[X]]\to (\text{gr }R)[Y],\underset{n\in\mathbb{N}}{\sum}{r_nX^n}+F_{\alpha^+}^{\lambda}R[[X]]\mapsto\underset{n\in\mathbb{N}}{\sum}{\left(r_n+F_{(\alpha-\lambda n)^+}R\right)Y^n}$.\\

\noindent Then $\Theta_{\lambda}$ is well-defined because since $\lambda>0$, it follows from Lemma \ref{attained} that the set 

\begin{center}
$\mathcal{X}:=\left\{n\in\mathbb{N}:v_{\lambda}\left(\underset{n\in\mathbb{N}}{\sum}{r_nX^n}\right)=v(r_n)+\lambda n\right\}$
\end{center}

\noindent is non-empty and finite. So if $v\left(\underset{n\in\mathbb{N}}{\sum}{r_nX^n}\right)\geq \alpha$ then for each $n\in\mathbb{N}$, $v(r_n)+\lambda n\geq\alpha$ so $r_n\in F_{(\alpha-\lambda n)}R$, and $r_n+F_{(\alpha-\lambda n)^+}R\neq 0$ if and only if $v_{\lambda}\left(\underset{n\in\mathbb{N}}{\sum}{r_nX^n}\right)=\alpha$ and $n\in \mathcal{X}$, thus $\underset{n\in\mathbb{N}}{\sum}{\left(r_n+F_{(\alpha-\lambda n)^+}R\right)Y^n}$ is a finite sum, so it lies in $(\text{gr }R)[Y]$.  And of course the map extends to the direct sum of the graded pieces.\\

\noindent To show that $\Theta_{\lambda}$ is a ring homomorphism, it suffices to show that for any $A,B\in\text{gr}_{\lambda}\text{ }R[[X]]$ homogeneous, $\Theta_{\lambda}(A+B)=\Theta_{\lambda}(A)+\Theta_{\lambda}(B)$ if $A$ and $B$ have the same degree, and $\Theta_{\lambda}(AB)=\Theta_{\lambda}(A)\Theta_{\lambda}(B)$ regardless of degree. 

Suppose that $A=f+F_{\alpha^+}^{\lambda}R[[X]]$ and $B=g+F_{\beta^+}^{\lambda}R[[X]]$ for some $\alpha,\beta>0$, and suppose that $f=\underset{n\in\mathbb{N}}{\sum}{r_nX^n}$ and $g=\underset{n\in\mathbb{N}}{\sum}{s_nX^n}$.\\

\noindent If $\alpha=\beta$ then $\Theta_{\lambda}(A+B)=\Theta_{\lambda}\left(\underset{n\in\mathbb{N}}{\sum}{(r_n+s_n)X^n}+F_{\alpha^+}^{\lambda}R[[X]]\right)=\underset{n\in\mathbb{N}}{\sum}{\left(r_n+s_n+F_{(\alpha-\lambda n)^+}R\right)Y^n}=\underset{n\in\mathbb{N}}{\sum}{\left(r_n+F_{(\alpha-\lambda n)^+}R\right)Y^n}+\underset{n\in\mathbb{N}}{\sum}{\left(s_n+F_{(\alpha-\lambda n)^+}R\right)Y^n}=\Theta_{\lambda}(A)+\Theta_{\lambda}(B)$.\\

\noindent Also, $\Theta_{\lambda}(AB)=\Theta_{\lambda}\left(\underset{n\in\mathbb{N}}{\sum}{\left(\underset{i+j=n}{\sum}{r_is_j}\right)X^n}+F_{(\alpha+\beta)^+}^{\lambda}R[[X]]\right)=\underset{n\in\mathbb{N}}{\sum}{\left(\underset{i+j=n}{\sum}{r_is_j}+F_{(\alpha+\beta-\lambda n)^+}R\right)Y^n}=\underset{n\in\mathbb{N}}{\sum}{\left(\underset{i+j=n}{\sum}{r_is_j}+F_{(\alpha-\lambda i+\beta-\lambda j)^+}R\right)Y^n}=\left(\underset{i\in\mathbb{N}}{\sum}{\left(r_i+F_{(\alpha-\lambda i)^+}R\right)Y^i}\right)\cdot\left(\underset{j\in\mathbb{N}}{\sum}{\left(s_j+F_{(\beta-\lambda j)^+}R\right)Y^j}\right)=\Theta_{\lambda}(A)\cdot\Theta_{\lambda}(B)$.\\

\noindent To show that $\Theta_{\lambda}$ is injective, suppose that $\Theta_{\lambda}\left(\underset{\alpha\geq 0}{\sum}{\left(\underset{n\in\mathbb{N}}{\sum}{r_{\alpha,n}X^n}+F_{\alpha^+}^{\lambda}R[[X]]\right)}\right)=0$. Then $\underset{\alpha\geq 0}{\sum}{\underset{n\in\mathbb{N}}{\sum}{\left(r_{\alpha,n}+F_{(\alpha-\lambda n)^+}R\right)Y^n}}=0$, so $\underset{\alpha\geq 0}{\sum}{r_{\alpha,n}+F_{(\alpha-\lambda n)^+}R}=0$ for all $n$. 

Thus $r_{\alpha,n}+F_{(\alpha-\lambda n)^+}R=0$ for all $\alpha,n$, so $v(r_{\alpha,n})+\lambda n>\alpha$. So by Lemma \ref{attained} this means that $v_{\lambda}\left(\underset{n\in\mathbb{N}}{\sum}{r_{\alpha,n}X^n}\right)>\alpha$ for all $\alpha$, and hence $\underset{\alpha\geq 0}{\sum}{\left(\underset{n\in\mathbb{N}}{\sum}{r_{\alpha,n}X^n}+F_{\alpha^+}^{\lambda}R[[X]]\right)}=0$.\\

\noindent Finally, if $\underset{n\in\mathbb{N}}{\sum}{A_nY^n}\in(\text{gr }R)[Y]$, with $A_n=r_n+F_{\alpha_n}R$ for each $n$, and we may assume that $v(r_n)=\alpha_n$. For each $\alpha>0$ let $B_{\alpha}=\{n\in\mathbb{N}:\alpha_n+\lambda n=\alpha\}$, then $\Theta_{\lambda}\left(\underset{\alpha>0}{\sum}\left({\underset{n\in B_{\alpha}}{\sum}{r_{n}X^n}+F_{\alpha^+}^{\lambda}R[[X]]}\right)\right)=\underset{n\in\mathbb{N}}{\sum}{A_nY^n}$, thus $\Theta_{\lambda}$ is surjective.\\

\noindent A separated filtration on a ring is a valuation if and only if its associated graded ring is a domain, thus gr $R$ is a domain, and hence $($gr $R)[Y]$ is a domain. So gr$_{\lambda}$ $R[[X]]\cong ($gr $R)[Y]$ is a domain and $v_{\lambda}$ is a valuation.\end{proof}

\noindent Note that this proposition strongly depends on the hypothesis that $\lambda>0$, and there is no similar isomorphism when $\lambda=0$. We now want to use this result to prove that $v_0$ is also a valuation.\\

\noindent Fix $f\in R[[X]]$ and define a map $\chi_{f}:[0,\infty)\to [0,\infty),\lambda\mapsto v_{\lambda}(f)$.

\begin{proposition}\label{convergence}
$\chi_f$ is monotonic increasing on $[0,\infty)$, and is continuous at 0.
\end{proposition}

\begin{proof}

If $f=\underset{n\in\mathbb{N}}{\sum}{r_nX^n}$ and $0\leq\lambda_1\leq\lambda_2$, then $\chi_f(\lambda_j)=\inf\{v(r_n)+\lambda_jn:n\in\mathbb{N}\}$ for $j=1,2$.\\

\noindent So, if $\chi_f(\lambda_2)<\chi_f(\lambda_1)$ then there exists $n\in\mathbb{N}$ such that $v(r_n)+\lambda_2n<v(r_m)+\lambda_1m$ for all $m$. In particular, $v(r_n)+\lambda_2n<v(r_n)+\lambda_1n$ and hence $(\lambda_2-\lambda_1)n<0$, which is impossible since $\lambda_2\geq\lambda_1$.\\

\noindent Therefore $\chi_f(\lambda_1)\leq\chi_f(\lambda_2)$ and hence $\chi_f$ is monotonic increasing. To prove that $\chi_f$ is continuous at 0, we need to prove that for all $\varepsilon>0$, $\vert\chi_f(\lambda)-\chi_f(0)\vert<\varepsilon$ for sufficiently small $\lambda$. 

Since $\chi_f$ is monotonic increasing, we know that $\vert\chi_f(\lambda)-\chi_f(0)\vert=\chi_f(\lambda)-\chi_f(0)$ for all $\lambda$, and if $\chi_f(\lambda')-\chi_f(0)<\varepsilon$ for some $\lambda'$, it follows that $\chi_f(\lambda)-\chi_f(0)\leq\chi_f(\lambda')-\chi_f(0)<\varepsilon$ for all $\lambda\leq\lambda'$. Therefore, it suffices only to prove that for all $\varepsilon>0$, $\chi_f(\lambda)-\chi_f(0)<\varepsilon$ for some $\lambda>0$.\\

\noindent Suppose for contradiction that there exists $\varepsilon>0$ such that $\chi_f(\lambda)\geq\varepsilon+\chi_f(0)$ for all $\lambda$, i.e. $v(r_n)+\lambda n\geq\varepsilon+\chi_f(0)$ for all $\lambda>0$, $n\in\mathbb{N}$.\\

\noindent But $\chi_f(0)=v_0(f)=\inf\{v(r_n):n\in\mathbb{N}\}$, so since $\varepsilon>0$, there exists $n\in\mathbb{N}$ such that $\chi_f(0)\leq v(r_n)<\varepsilon+\chi_f(0)$, and $n>0$ since otherwise $v(r_0)+\lambda 0=v(r_0)<\varepsilon+\chi_f(0)$. So choose $\lambda>0$ with $\lambda<\frac{\varepsilon+\chi_f(0)-v(r_n)}{n}$ and it follows that $v(r_n)+\lambda n<\varepsilon+\chi_f(0)$ -- contradiction.\end{proof}

\begin{corollary}\label{valuation2}
$v_0$ is a valuation on $R[[X]]$.
\end{corollary}

\begin{proof}

Given $f,g\in R[[X]]$, we need to show that $v_0(fg)=v_0(f)+v_0(g)$. Using Proposition \ref{valuation1}, we know that for all $\lambda>0$, $v_{\lambda}$ is a valuation, so $v_{\lambda}(fg)=v_{\lambda}(f)+v_{\lambda}(g)$. Also, using Proposition \ref{convergence}, we know that $\chi_f$, $\chi_g$ and $\chi_{fg}$ are continuous at zero, so it follows that: \\

\noindent $v_0(fg) =\chi_{fg}(0) =\underset{\lambda\rightarrow 0}{\lim}{\chi_{fg}(\lambda)}\text{ (since $\chi_{fg}$ is continuous at 0)}$\\ 

$=\underset{\lambda\rightarrow 0}{\lim}{\text{ }v_{\lambda}(fg)} =\underset{\lambda\rightarrow 0}{\lim}{\text{ }v_{\lambda}(f)+v_{\lambda}(g)}\text{ (since $v_{\lambda}$ is a valuation for $\lambda>0$)}$\\

$=\underset{\lambda\rightarrow 0}{\lim}{v_{\lambda}(f)}+\underset{\lambda\rightarrow 0}{\lim}{v_{\lambda}(g)}=\underset{\lambda\rightarrow 0}{\lim}{\chi_{f}(\lambda)}+\underset{\lambda\rightarrow 0}{\lim}{\chi_{g}(\lambda)}=\chi_f(0)+\chi_g(0)=v_0(f)+v_0(g)$.\end{proof}

\section{Incoherence of $R[[X]]$}

From now on, let $F$ be a field, let $v:F\to\mathbb{R}\cup\{\infty\}$ be a valuation, and let $R:=\{x\in F:v(x)\geq 0\}$. Then $R$ is a valuation ring of rank 1, and we will assume that the value group $v(F\backslash 0)$ is a proper, dense subgroup of $\mathbb{R}$. Using Corollary \ref{valuation2}, the valuation $v$ on $R$ extends to a valuation $v_0$ of $R[[X]]$ given by $v_0\left(\underset{n\in\mathbb{N}}{\sum}{r_nX^n}\right)=\inf\{v(r_n):n\in\mathbb{N}\}$.\\

\noindent We know that $R$ is a coherent ring, but it was proved in \cite[Corollary Section 3]{rank-1} that $R[[X]]$ is never coherent, and in this section, we will give an alternative (albeit similar) proof using the valuation theory we have developed.\\

\noindent For convenience, we will write $v(F)$ and $v(R)$ to mean $v(F\backslash 0)$ and $v(R\backslash 0)$ respectively. Since $v(F)$ is a proper, dense subset of $\mathbb{R}$, it follows that $v(R)=\{\alpha\in v(F):\alpha\geq 0\}$ is a dense subset of $\mathbb{R}_{\geq 0}$, not equal to $\mathbb{R}_{\geq 0}$. Therefore, there exists an element $\alpha\in\mathbb{R}_{\geq 0}$ such that $\alpha\notin v(R)$, and there exists a sequence of elements $\alpha_n\in v(R)$ such that $\alpha_n\geq\alpha_{n+1}>\alpha$ for all $n$ and $\alpha_n\rightarrow\alpha$ as $n\rightarrow\infty$.\\

\noindent Fix elements $r_n\in R$ such that $v(r_n)=\alpha_n$ for all $n$, and let $f:=\underset{n\in\mathbb{N}}{\sum}{r_nX^{n}}\in R[[X]]$. Then $v_0(f)=\inf\{v(r_n):n\in\mathbb{N}\}=\inf\{\alpha_n:n\in\mathbb{N}\}=\alpha$.\\

\noindent Now, choose any $r\in R$ with $v(r)>\alpha$, and consider the ideal $J:=R[[X]]f+R[[X]]r$. This is clearly a finitely generated ideal of $R[[X]]$. 

\begin{theorem}\label{B}

$J$ is not finitely presented, and hence $R[[X]]$ is not a coherent ring.

\end{theorem}

\begin{proof}

First, consider the ideal:

\[
I:=\{g\in R[[X]]:v_0(g)\geq v(r)-\alpha\}.
\]

\noindent Then if $h\in R[[X]]f\cap R[[X]]r$ then $h=gf=yr$ for some $g,y\in R[[X]]$. So since $v_0$ is a valuation, we have that $v_0(h)=v_0(g)+v_0(f)=v_0(y)+v_0(r)$, so since $v_0(f)=\alpha$ and $v_0(r)=v(r)$, it follows that $v_0(g)=v_0(y)+v(r)-\alpha\geq v(r)-\alpha$, and hence $g\in I$.

Conversely, if $g\in I$ then $v_0(gf)=v_0(g)+\alpha\geq v(r)-\alpha+\alpha=v(r)$, so if $gf=\underset{n\in\mathbb{N}}{\sum}{t_nX^{n}}$ then $v(t_n)\geq v(r)$ for all $n\in\mathbb{N}$. So since $R$ is a valuation ring, this means that $r$ divides $t_n$ for all $n\in\mathbb{N}$, which means that $gf=ry$, where $y=\underset{n\in\mathbb{N}}{\sum}{(r^{-1}t_n)X^{n}}\in R[[X]]$, so $gf\in R[[X]]f\cap R[[X]]r$.\\

\noindent Therefore, $R[[X]]f\cap R[[X]]r=If$, so if the intersection is finitely generated, then $If$ is finitely generated. But since $R[[X]]$ is a domain, it follows that if $\{g_1f,\dots,g_mf\}$ is a generating set for $If$ then $\{g_1,\dots,g_m\}$ is a generating set for $I$. So if $g_i:=\underset{n\in\mathbb{N}}{\sum}{r_{i,n}X^n}$ then let $\beta:=\min\{v(r_{i,0}):i=1,\dots,m\}$, and it follows that for every $g=\underset{n\in\mathbb{N}}{\sum}{r_{n}X^n}\in I$, $v(r_0)\geq\beta$, and clearly $\beta\geq\min\{\inf\{v(r_{i,n}):n\in\mathbb{N}\}:i=1,\dots,m\}\geq v(r)-\alpha$.

But since $\alpha\notin v(R)$, it follows that $v(r)-\alpha\notin v(R)$, so since $\beta\in v(R)$ we see that $\beta>v(r)-\alpha$. But $v(R)$ is dense in $\mathbb{R}_{\geq 0}$, so we can find $t\in R$ such that $\beta>v(t)>v(r)-\alpha$, and hence $t=tX^0\in I$ but $v_0(t)=v(t)<\beta$ -- contradiction.\\

\noindent Therefore, $R[[X]]f\cap R[[X]]r$ is not finitely generated. So, consider the exact sequence $0\to R[[X]]f\cap R[[X]]r\to R[[X]]^2\to J\to 0$, then since the kernel is not finitely generated it follows from \cite[Lemma 2.1.1]{Glaz} that $J$ is not finitely presented.\end{proof}

\section{Localisations of $R[[X]]$}

\noindent In this section, we will use our valuation theory to explore a localisation of $R[[X]]$ with useful properties. Let $U:=\{f\in R[[X]]:v_0(f)=0\}$, then since $v_0$ is a valuation, $U$ is a multiplicatively closed subset of $R[[X]]$. So since $R[[X]]$ is commutative, we can consider the localisation $T:=R[[X]]_U$ of $R[[X]]$ at $U$.

\begin{proposition}\label{loc-valuation}
$T$ is a valuation ring.
\end{proposition}

\begin{proof}

Let $V:=\{ru:r\in R\backslash\{0\},u\in U\}$. Then $V$ is a multiplicatively closed subset of $R[[X]]$, so let $K:=R[[X]]_V$, and since $U\subseteq V$ it is clear that $T$ is a subring of $K$.\\

\noindent By Lemma \ref{localisation}, the valuation $v_0$ extends uniquely to any localisation of $R[[X]]$, so it follows that $K$ and $T$ both carry valuations that restrict to $v_0$ on $R[[X]]$. We will prove that $K$ is a field and that $T:=\{x\in K:v_0(x)\geq 0\}$, and it will follow from Lemma \ref{val-ring} that $T$ is a valuation ring.\\

\noindent To prove that $K$ is a field, it suffices to show that every non-zero element of $R[[X]]$ is a unit in $K$. Suppose that $0\neq f\in R[[X]]$ and $v_0(f)=\alpha\in\mathbb{R}_{\geq 0}$, we will prove that $f$ is a unit in $K=R[[X]]_V$:\\

\noindent Firstly, if $\alpha\in v(R)$ then $\alpha=v(r)$ for some $r\in R\backslash 0$, and if $f=\underset{n\in\mathbb{N}}{\sum}{r_nX^n}$, then $v(r)=\alpha=v_0(f)=\inf\{v(r_n):n\in\mathbb{N}\}$, and hence $v(r_n)\geq v(r)$ for all $n$ and $r^{-1}r_n\in R$. So let $u:=\underset{n\in\mathbb{N}}{\sum}{(r^{-1}r_n)X^n}\in R[[X]]$, then $f=ru$, and if $v_0(u)>0$ then there exists $\varepsilon>0$ such that $v(r^{-1}r_n)\geq\epsilon$ for all $n$, and hence $v(r_n)\geq v(r)+\epsilon$, so $v_0(f)\geq v(r)+\epsilon>v(r)=\alpha=v_0(f)$ -- contradiction. Therefore $v_0(u)=0$ and $u\in U$, so $f=ru\in V$ is a unit in $K=R[[X]]_V$.

On the other hand, if $\alpha\notin v(R)$, then choose $\beta\in v(R)$ with $\beta>\alpha$. Then since $v(R)$ is dense in $\mathbb{R}_{\geq 0}$, there exists a sequence $(\gamma_n)$ in $v(R)$ such that $\gamma_n\geq\gamma_{n+1}>\beta-\alpha$ for each $n$ and $\gamma_n\rightarrow\beta-\alpha$ as $n\rightarrow\infty$. So choose $s_n\in R$ such that $v(s_n)=\gamma_n$ for each $n$, and let $g:=\underset{n\in\mathbb{N}}{\sum}{s_nX^n}$, then $v_0(g)=\inf\{v(s_n):n\in\mathbb{N}\}=\inf\{\gamma_n:n\in\mathbb{N}\}=\beta-\alpha$. Therefore, $v_0(fg)=v_0(f)+v_0(g)=\alpha+\beta-\alpha=\beta\in v(R)$, so by the above, $fg$ is a unit in $K$, and hence $f$ is a unit in $K$.\\

\noindent Finally, it is clear that for any $x=fu^{-1}\in T$, $v_0(x)=v_0(f)-v_0(u)=v_0(f)\geq 0$, so it remains to prove that if $x\in K$ and $v_0(x)\geq 0$ then $x\in T$. So, $x=f(ru)^{-1}$ for some $r\in R$, $u\in U$, and $v_0(x)=v_0(f)-v_0(ru)\geq 0$, i.e. $v_0(f)\geq v(r)$. So, if $f=\underset{n\in\mathbb{N}}{\sum}{r_nX^n}$, then $v_0(f)=\inf\{v(r_n):n\in\mathbb{N}\}\geq v(r)$, so $v(r_n)\geq v(r)$ for all $n$ and $r^{-1}r_n\in R$. So let $g:=\underset{n\in\mathbb{N}}{\sum}{(r^{-1}r_n)X^n}\in R[[X]]$, then $f=rg$, so $x=f(ru)^{-1}=fr^{-1}u^{-1}=gu^{-1}\in R[[X]]_U=T$ as required.\end{proof}

\noindent Now, recall from \cite{Glaz} that if $S$ is a commutative ring and $M$ is an $S$-module, then a submodule $N$ of $M$ is a \emph{pure submodule} if for any $S$-module $L$, the natural map $N\otimes_S L\to M\otimes_S L$ is injective. It follows from \cite[Theorem 1.2.14(5)]{Glaz} that if $M$ is a flat $S$-module and $N$ is a pure submodule then $M/N$ is a flat $S$-module.

\begin{proposition}\label{pure}

$R[[X]]$ is a pure $R$-submodule of $T$, and more generally, for any indexing set $I$, the direct product $R[[X]]^I$ is a pure $R$-submodule of $T^I$.

\end{proposition}

\begin{proof}

Using \cite[Theorem 1.2.14(5)]{Glaz}, we only need to prove that for any finitely generated ideal $J$ of $R$, $JT^I\cap R[[X]]^I=JR[[X]]^I$. Since $R$ is a valuation ring, $J=aR$ for some $a\in R$, so we need only prove that $aT^I\cap R[[X]]^I=aR[[X]]^I$.\\

\noindent In fact, if we proved that $R[[X]]$ is a pure submodule of $T$, i.e. $aT\cap R[[X]]=aR[[X]]$, then if $(x_i)_{i\in I}\in aT^I\cap R[[X]]^I$ then $x_i\in aT\cap R[[X]]=aR[[X]]$ for each $i$, and hence $(x_i)_{i\in I}\in (aR[[X]])^I=aR[[X]]^I$, and it follows that $R[[X]]^I$ is a pure $R$-submodule of $T^I$.\\

\noindent So, suppose $f\in aT\cap R[[X]]$, then $f\in R[[X]]$ and $f=agu^{-1}$ for some $g\in R[[X]]$, $u\in U$, thus $ag=fu$. Since $v_0(u)=0$ we see that $v_0(f)=v_0(f)+v_0(u)=v_0(fu)=v_0(ag)=v_0(a)+v_0(g)\geq v_0(a)=v(a)$.\\

\noindent Therefore, if $f=\underset{n\in\mathbb{N}}{\sum}{r_nX^n}$ then $v_0(f)=\inf\{v(r_n):n\in\mathbb{N}\}\geq v(a)$, so $v(r_n)\geq v(a)$ for all $n$, which means that $a^{-1}r_n\in R$ for all $n$. So let $h:=\underset{n\in\mathbb{N}}{\sum}{(a^{-1}r_n)X^n}\in R[[X]]$, and $f=ah\in aR[[X]]$ as required.\end{proof}

\section{Flat dimension of $R[[X]]$-modules}

\noindent In this section, we will prove our main result, providing an example of an $R[[X]]$-module $C$ that is flat over $R$, but has flat dimension at least 2 over $R[[X]]$, thus contradicting \cite[Corollary 3.2]{flaw}. The proof is heavily inspired by the proof of \cite[Theorem 7.2.2]{Glaz}\\

\noindent Using Theorem \ref{B}, we know that $R[[X]]$ is not a coherent ring, and therefore by \cite[Theorem 2.3.2(4)]{Glaz} there exists an indexing set $I$ such that the direct product $R[[X]]^I$ is not flat over $R[[X]]$. As in the previous section, we let $T:=R[[X]]_{U}$ where $U=\{f\in R[[X]]:v_0(f)=0\}$, and let $C:=T^I/R[[X]]^I$.

\begin{lemma}\label{flat/R}

$T^I$ is flat over $R[[X]]$ and $R$, and $C$ is flat over $R$.

\end{lemma}

\begin{proof}

We know that $T$ is a valuation ring by Proposition \ref{loc-valuation}, and hence it is coherent. Therefore, by \cite[Theorem 2.3.2(4)]{Glaz}, $T^I$ is a flat $T$-module. But $T$ is a localisation of $R[[X]]$, and hence $T$ is flat over $R[[X]]$ by Lemma \ref{localisation}, thus $T^I$ is flat over $R[[X]]$. 

Moreover, since $R$ is coherent and $R[[X]]\cong\underset{n\in\mathbb{N}}{\prod}{R}$ as an $R$-module, $R[[X]]$ is a flat $R$-module by \cite[Theorem 2.3.2(4)]{Glaz}, and hence $T^I$ is a flat $R$-module.\\

\noindent Since $T^I$ is flat over $R$, and $R[[X]]^I$ is a pure $R$-submodule of $T^I$ by Proposition \ref{pure}, it follows from \cite[Theorem 1.2.14(5)]{Glaz} that $C=T^I/R[[X]]^I$ is a flat $R$-module.\end{proof}

\noindent We are now ready to prove our main theorem:\\

\noindent\emph{Proof of Theorem \ref{A}.} We know that $C$ is flat over $R$ by Lemma \ref{flat/R}, so clearly f.dim$_R(C)=0$. Let us suppose, for contradiction, that f.dim$_{R[[X]]}(C)\leq 1$, i.e. $Tor^{R[[X]]}_2(C,N)=0$ for all $R[[X]]$-modules $N$. Then there exists a long exact sequence:\\

$0\to Tor^{R[[X]]}_1(R[[X]]^I,N)\to Tor^{R[[X]]}_1(T^I,N)\to Tor^{R[[X]]}_1(C,N)$\\

$\to R[[X]]^I\otimes_{R[[X]]}N\to T^I\otimes_{R[[X]]}N\to C\otimes_{R[[X]]}N\to 0$\\

\noindent But we know from Lemma \ref{flat/R} that $T^I$ is flat over $R[[X]]$, and hence $Tor^{R[[X]]}_1(T^I,N)=0$, so it follows that $Tor^{R[[X]]}_1(R[[X]]^I,N)=0$. Since this is true for all $R[[X]]$-modules $N$ it follows that $R[[X]]^I$ is flat over $R[[X]]$, contradicting our original assumption.\qed\\

\noindent In fact, we could use the same argument to show that for all $i$ and $N$,  $Tor^{R[[X]]}_{i+1}(C,N)\cong Tor^{R[[X]]}_i(R[[X]]^I,N)$, and hence f.dim$_{R[[X]]}(C)=$ f.dim$_{R[[X]]}(R[[X]]^I)+1$. So if we could show that there exists an indexing set $I$ such that f.dim$_{R[[X]]}(R[[X]]^I)>$ w.dim$(R)$, it would follow that w.dim$(R[[X]])>$ w.dim$(R)+1$, thus showing that the weak Hilbert's syzygy theorem fails to hold for power series over coherent rings, even when the coefficient ring is a valuation ring of rank 1.

\end{document}